\documentclass[12pt, oneside, a4paper]{article}

\usepackage{amsmath}
\usepackage{amsfonts}
\usepackage{amssymb}
\usepackage{amsthm,mathrsfs}
\usepackage{enumerate}
\usepackage{graphicx}
\newtheorem{theorem}{Theorem}[section]
\newtheorem{lemma}[theorem]{Lemma}

\newtheorem{corollary}[theorem]{Corollary}

\theoremstyle{definition}
\newtheorem{remark}[theorem]{Remark}

\title{\textbf{A hook formula for eigenvalues of $k$-point fixing graph}}
\author{Mahdi Ebrahimi\footnote{ m.ebrahimi.math@ipm.ir}
 \\
 {\small\em  School of Mathematics, Institute for Research in Fundamental Sciences (IPM)},\\{\small\em P.O. Box: 19395--5746, Tehran, Iran}\\
\\
}
\date{}

\begin{document}

\maketitle


\begin{abstract}
Let $S_n$ denote the symmetric group on $n$ letters. The \textit{$k$-point fixing graph} $\mathcal{F}(n,k)$ is defined to be the graph with vertex set $S_n$ and two vertices $g,h$ of $\mathcal{F}(n,k)$ are joined by an edge, if and only if $gh^{-1}$ fixes exactly $k$ points. Ku, Lau and Wong [Cayley graph on symmetric group generated by elements fixing $k$ points, Linear Algebra Appl. 471 (2015) 405-426] obtained a recursive formula for the eigenvalues of $\mathcal{F}(n,k)$.
 In this paper, we use objects called excited diagrams  defined as certain generalizations of skew shapes and
 derive an explicit formula for the eigenvalues of Cayley graph $\mathcal{F}(n,k)$. Then we apply this formula and show that the eigenvalues of $\mathcal{F}(n,k)$ are in the interval $[\frac{-|S(n,k)|}{n-k-1}, |S(n,k)|]$, where $S(n,k)$ is the set of elements $\sigma$ of $S_n$ such that $\sigma$ fixes exactly $k$ points.
 \end{abstract}

\section{Introduction}
$\noindent$ For a graph $\Gamma$, the \textit{eigenvalues} of $\Gamma$ is the eigenvalues of its adjacency matrix. The study of eigenvalues of graphs is an important part of modern graph theory. In particular, eigenvalues of Cayley graphs have attracted increasing attention due to their prominent roles in algebraic graph theory and applications in many areas such as expanders \cite{192,279}, chemical graph theory \cite{377} and quantum computing \cite{58,365}.
Let $G$ be a finite group and $S$
 be an inverse closed subset of $G$ with $1 \notin S$.
  The \textit{Cayley graph} $\Gamma(G,S)$ is the graph which has the elements of
   $G$ as its vertices and two vertices $u,\nu \in G$
    are joined by an edge if and only if $\nu=au$, for some $a\in S$. A Cayley graph $\Gamma(G,S)$ is called \textit{normal} if $S$ is closed under conjugation with elements of $G$. Also $\Gamma(G,S)$ is called \textit{integral} if its eigenvalues are all integers.

 Suppose $k$ and $n$ are positive integers. For $k\leq n$, a \textit{$k$-permutation} of $[n]:=\{1,2,\dots,n\}$ is an injective function from $[k]$ to $[n]$. Let $1\leq r\leq k\leq n$. The \textit{$(n,k,r)$-arrangement graph} $A(n,k,r)$ has all the $k$-permutations of $[n]$ as vertices and two $k$-permutations are adjacent if they differ in exactly $r$ positions. Note that $A(n,k,r)$ is a regular graph \cite{94}. The family of the arrangement graphs $A(n,k,1)$ was first introduced in \cite{9} as an interconnection network model for parallel computation. A relation between the eigenvalues of $A(n,k,r)$ and certain Cayley graphs was given in \cite{94}.

Let $n$ be a positive integer. Given an integer $k$ with $0\leq k\leq n-1$, let $S(n,k)$ be the set of elements $\sigma$ of $S_n$ such that $\sigma$ fixes exactly $k$ points in $[n]$. The $k$-point-fixing graph is defined \cite{237} to be the Cayley graph $\mathcal{F}(n,k):=\Gamma(S_n, S(n,k))$, that is, two vertices $\sigma, \tau$ are adjacent if and only if  $\sigma \tau^{-1}$ fixes exactly $k$ points. Note that the $k$-point fixing graph is also a kind of arrangement graph, i.e., $\mathcal{F}(n,k)=A(n,n,n-k)$. Since $S(n,k)$ is closed under conjugation, all k-point-fixing graphs are integral \cite[Corollary 1.2]{94}. Ku, Lau and Wong \cite{237} obtained a recursive formula for the eigenvalues of $\mathcal{F}(n,k)$, and using this formula, they determined the signs of the eigenvalues of the 1-point-fixing graph $\mathcal{F}(n,1)$. In \cite{239}, they obtained exact values of some eigenvalues of $\mathcal{F}(n,1)$. Also Renteln \cite{338} gave several interesting formulas for the eigenvalues of $\mathcal{F}(n,0)$. In this paper, we wish to obtain an explicit formula for eigenvalues of $\mathcal{F}(n,k)$.

 It is well known that the eigenvalues of a normal Cayley graph  $\Gamma(G,S)$ can be expressed in terms of the irreducible characters of $G$ \cite[p.235]{eigen}.

\begin{theorem}\label{eigen}(\cite{2}, \cite{6}, \cite{17}, \cite{19})
The eigenvalues of a normal Cayley graph $\Gamma(G,S)$
are given by $\eta_\chi=\frac{1}{\chi(1)}\sum_{a\in S}\chi(a)$ where
 $\chi$ ranges over all the complex irreducible characters of $G$. Moreover,
  the multiplicity of $\eta_{\chi}$ is $\chi(1)^2$.
\end{theorem}

A partition is a weakly decreasing finite sequence of positive integers $\lambda=(\lambda_1, \dots, \lambda_l)$. We call $|\lambda|=\lambda_1+\dots+\lambda_l$ the \textit{size} of $\lambda$, and $l=l(\lambda)$ the \textit{length} of $\lambda$. The notation $\lambda \vdash n$ is used for a partition $\lambda$ of a positive integer $n$. The \textit{diagram} of $\lambda$ is $[\lambda]=\{(i,j)|\;1\leq i\leq l(\lambda)$, $1\leq j\leq \lambda_i\}$. We call the elements of $[\lambda]$ the \textit{cells} of $\lambda$. For partitions $\mu$ and $\lambda$, we say that $mu$ is \textit{contained} in $\lambda$, $\mu\subseteq \lambda$, if $[\mu]\subseteq [\lambda]$. We say that $\lambda/\mu$ is a \textit{skew shape} of size $|\lambda/\mu|=|\lambda|-|\mu|$ and the diagram of $\lambda/\mu$ is $[\lambda/\mu]=[\lambda]\backslash [\mu]$.

 The \textit{conjugate} of a partition $\lambda$ is the partition $\lambda^\prime$ whose diagram is the transpose of $[\lambda]$; in other words, $\lambda^\prime_j=max \{i|\;\lambda_i\geq j\}$.
    The \textit{hook length} $h_{\lambda}(u):=\lambda_i-i+\lambda^\prime_j-j+1$ of a cell $u=(i,j)\in [\lambda]$ is the number of cells directly to the right  and directly below $u$ in $[\lambda]$.

   Excited diagrams  defined as certain generalizations of skew shapes play an important role in combinatorics and representation theory of symmetric groups.
    Excited diagrams were introduced by Ikeda and Naruse \cite{18}, and in a slightly different form independently by Kreiman \cite{29}, \cite{30} and Knutson, Miller and Young \cite{21}.
     Let $\lambda/ \mu$ be a skew partition and $D$ be a subset of the Young diagram of $\lambda$.
      A cell $u=(i,j)\in D$ is called \textit{active} if $(i+1,j), (i,j+1)$ and $(i+1,j+1)$ are all in $[\lambda]\backslash D$.
      Let $u$ be an active cell of $D$, define $\alpha_u(D)$ to be the set obtained by replacing $(i,j)$ in $D$ by $(i+1,j+1)$.
      We call this replacement an \textit{excited move}. An \textit{excited diagram} of $\lambda/ \mu$ is a subdiagram of $\lambda$ obtained from the Young diagram of $\mu$ after a sequence of excited moves on active cells.
     For example, $(2^3,1)/(1^2)$ has three excited diagrams $\{(1,1),(2,1)\},\{(1,1),(3,2)\}$ and $\{(2,2),(3,2)\}$.  The set of excited diagrams of $\lambda/\mu$ is denoted by $\varepsilon(\lambda/\mu)$.
      Now we are ready to present our main result.

\begin{theorem}\label{main}
The eigenvalues of $\mathcal{F}(n,k)$ are given by
$$\eta_{\lambda}(k):= \sum_{\substack{\mu \vdash n-k\\ \mu \subseteq \lambda}}[\frac{\sum_{t=0}^{n-k}(-1)^{n-k-t}\sum_{D\in \varepsilon(\mu/(t))}\prod_{u\in D}h_{\mu}(u)}{\prod_{u\in [\mu]}h_{\mu}(u)}\sum_{E\in \varepsilon(\lambda/\mu)}\prod_{u\in E}h_{\lambda}(u)],$$
where $\lambda$ ranges over partitions of $n$.
   Moreover, the multiplicity of $\eta_{\lambda} (k)$ is
    $(\frac{n!}{{\prod}_{u\in [\lambda]} h_{\lambda}(u)})^2$.
\end{theorem}

\begin{remark}
The number of excited diagrams of skew shape $\lambda/ \mu$ is given by a determinant \cite[Corollary 3.7]{Mor1}, a polynomial in the parts of $\lambda$ and $\mu$.
\end{remark}

The \textit{complete transposition graph} (also known as the \textit{transposition network}) $T_n$ is the Cayley graph on $S_n$ with connection set consisting of all transpositions in $S_n$. Note that $T_n=\mathcal{F}(n,n-2)$.

\begin{corollary}\label{transposition}
For the transposition network $T_n$,
\begin{itemize}
\item[a)] The eigenvalues of $T_n$ are given by
$$\eta_{\lambda}(n-2)=\sum_{\substack{(i,i),(j,j+1)\in [\lambda]\\ i\leq j}}h_\lambda((i,i))h_\lambda((j,j+1))-\binom{n}{2},$$
where $\lambda$ ranges over partitions of $n$.
\item[b)] The multiplicity of an eigenvalue $m$ of $T_n$ is equal to
$$\sum\{(\frac{n!}{{\prod}_{u\in [\lambda]} h_{\lambda}(u)})^2|\;\lambda \vdash n,\;\sum_{\substack{(i,i),(j,j+1)\in [\lambda]\\ i\leq j}}h_\lambda((i,i))h_\lambda((j,j+1))=\binom{n}{2}+m\}.$$
\end{itemize}
\end{corollary}

Suppose $k$ and $n$ are non-negative integers with $k < n$. For every partition $\lambda \vdash n$, we define $M_k(\lambda):=max\{|\eta_\mu (0)|\;|\;\mu \vdash n-k\;and\;\mu\subseteq \lambda\}$. As another application of our main theorem, we can state the following result.

\begin{corollary}\label{baund}
For the Cayley graph $\mathcal{F}(n,k)$,
\begin{itemize}
\item[a)] If $\lambda\vdash n$, then $|\eta_\lambda(k)|\leq \binom{n}{k}M_k(\lambda)$.
\item[b)] The eigenvalues of $\mathcal{F}(n,k)$ are in the interval $[\frac{-|S(n,k)|}{n-k-1}, |S(n,k)|]$.
\item[c)] $\eta_{(n)}(k)=|S(n,k)|$ and $\eta_{(1^n)}(k)=-(n-k-1) \binom{n}{k}$.
\item[d)] If $k=n-2$ or $n-4$, then the least eigenvalue of $\mathcal{F}(n,k)$ is given by
$\eta=\frac{-|S(n,k)|}{n-k-1}$.
\item[e)] Suppose $\lambda=(m, 1^{n-m})$, for some positive integer $ m < n$. Then $\eta_\lambda(k)$ is equal to
$$\frac{(-1)^{n-k}\binom{n}{k}}{\binom{n-1}{m-1}}\sum_{s=max\{1,m-k\}}^{min\{m,n-k\}}[((-1)^s|S(s,0)|-\frac{n-k-s}{n-k}){\binom{n-k}{s}}{\binom{k}{m-s}}].$$
\end{itemize}
\end{corollary}
\section{Preliminaries}
$\noindent$   In this paper, all groups are assumed to be finite. We first state well-known results on the character theory of the symmetric groups; for a complete account, see \cite{GA}. We often represent a partition $\lambda$ by its \textit{Young tableau}, in which a cell $(i,j)\in [\lambda]$ is represented by a unit square in position $(i,j)$, and we again denote it by $[\lambda]$. It is well known that both the conjugacy classes of $\mathcal{S}_n$ and the irreducible characters of $\mathcal{S}_n$ are indexed by partitions $\lambda$ of $[n]$ (see \cite{GA}). The irreducible character indexed by $\lambda \vdash n$ may be identified with the Young tableau $[\lambda]$ of $\lambda$. The character value of $[\lambda]$ on the conjugacy class indexed by $\beta \vdash n$ is denoted by $[\lambda]\beta$.
 In representation theory of symmetric groups, the branching rule tells us how to restrict an ordinary irreducible representation from $S_n$ to $S_{n-1}$.
 \begin{lemma}\label{branching}\cite[Theorem 2.4.3]{GA}(branching rule)

If $\alpha=(\alpha_1,\alpha_2,\dots,\alpha_s )$ is a partition of $n$, then we have for the restriction of $[\alpha]$ to the stabilizer $S_{n-1}$ of the point $n$
$$[\alpha]\downarrow S_{n-1}=\sum_{\substack{i\\ \alpha_i>\alpha_{i+1}}}[\alpha^{i^-}],$$
where $[\alpha^{i^-}]$ is a diagram obtained  by taking a cell away from $i$'th row  in $[\alpha]$.
\end{lemma}
A generalization of branching rule is as follows:
\begin{lemma}\label{product}
Suppose $m$ and $n$ are positive integers with $m\leq n$. Then for every $\lambda \vdash n$,
$$[\lambda]\downarrow S_{m}=\sum_{\substack{\mu \vdash m\\ \mu\subseteq \lambda}} f^{\lambda/\mu}[\mu].$$
\end{lemma}

\begin{proof}
We do by induction on $k:=n-m$. If $k=1$, then using branching rule, it is clear. Let $t\geq 2$ be an integer. Now, we assume that the statement is true for every $k<t$ and we prove it for $k=t$. By induction hypothesis,
$$[\lambda]\downarrow S_{m+1}=\sum_{\substack{\tilde{\mu} \vdash m+1\\ \tilde{\mu}\subseteq \lambda}} f^{\lambda/\tilde{\mu}}[\tilde{\mu}].$$
Thus using branching rule,
\begin{align}
[\lambda]\downarrow S_{m}&=\sum_{\substack{\tilde{\mu} \vdash m+1\\ \tilde{\mu}\subseteq \lambda}} f^{\lambda/\tilde{\mu}}[\tilde{\mu}]\downarrow S_{m}\nonumber\\
&=\sum_{\substack{\tilde{\mu} \vdash m+1\\ \tilde{\mu}\subseteq \lambda}}[f^{\lambda/\tilde{\mu}}\sum_{\substack{i\\ \tilde{\mu}_i>\tilde{\mu}_{i+1}}}[\tilde{\mu}^{i^-}]].\nonumber
\end{align}
If $\mu\vdash m$ and $\mu \subseteq \lambda$, then obviously, $f^{\lambda/\mu}=\sum_{\mu \subseteq \tilde{\mu}} f^{\lambda/\tilde{\mu}}$. Hence
  $$[\lambda]\downarrow S_{m}=\sum_{\substack{\mu \vdash m\\ \mu\subseteq \lambda}} f^{\lambda/\mu}[\mu].$$
\end{proof}

 The standard Young tableaux  and skew shapes are central objects in enumerative and algebraic combinatorics.
A \textit{standard Young tableau} (or SYT for short) of shape $\lambda$ is a bijective map $T:[\lambda]\rightarrow \{1,\dots, |\lambda|\}$, $(i,j)\mapsto T_{ij}$,
satisfying $T_{ij}<T_{i,j+1}$ if $(i,j)$,$(i,j+1)\in [\lambda]$
and $T_{ij}<T_{i+1,j}$ if $(i,j)$,$(i+1,j)\in [\lambda]$. The number of SYT's of shape $\lambda$ is denoted by $f^{\lambda}$.
Analogously, if $\mu \subseteq \lambda$,
 we can define a \textit{standard Young tableau of skew shape} $\lambda/\mu$ as
 a map $T:[\lambda/\mu]\rightarrow \{1,\dots, |\lambda/ \mu|\}$, $(i,j)\mapsto T_{ij}$,
 satisfying $T_{ij}<T_{i,j+1}$ if $(i,j)$, $(i,j+1)\in [\lambda/\mu]$ and
  $T_{ij}<T_{i+1,j}$ if $(i,j)$, $(i+1,j)\in [\lambda/\mu]$.
   The number of SYT's of shape $\lambda/\mu$ is denoted by $f^{\lambda/ \mu}$.

\begin{lemma}\label{hlf}(\cite{hlf})
Let $\lambda$ be a partition of $n$. We have:
$$f^\lambda=\frac{n!}{\prod_{u\in [\lambda]} h(u)}.$$
This formula also gives dimensions of the irreducible representation corresponding to $\lambda$.
\end{lemma}

\begin{lemma}\label{nhlf}(\cite{nhlf})( Naruse's formula)
Let $\lambda$, $\mu$ be partitions such that $\mu\subseteq \lambda$. We have:
$$f^{\lambda/\mu}=|\lambda/\mu|!\sum_{D\in \varepsilon(\lambda/\mu)}\prod_{u\in [\lambda]\backslash D}\frac{1}{h_\lambda(u)}.$$
\end{lemma}

\begin{remark}
There are an algebraic and a combinatorial proof of Naruse's formula in \cite{Mor1}. Also Konvalinka \cite{european} gave a simple bijection that proves an equivalent recursive version of Naruse's formula.
\end{remark}

 There exists an interesting formula to evaluation the eigenvalues of $\mathcal{F}(n,0)$ as follows:

\begin{lemma}\label{rentel}(\cite[Theorem 3.2]{338})
The eigenvalues of $\mathcal{F}(n,0)$ are given by
$$\eta_\lambda(0)=\sum_{k=0}^n(-1)^{n-k}\frac{n!}{(n-k)!}\frac{f^{\lambda/(k)}}{f^\lambda},$$
where $\lambda$ runs over all partitions of $n$.
\end{lemma}

\begin{lemma}\label{ku}(\cite[Theorem 7.1]{338})
The least eigenvalue of the adjacency matrix of the graph $\mathcal{F}(n,0)$ is given by
$$\eta=\frac{-|S(n,0)|}{n-1}.$$
\end{lemma}

\section{Proof of main results}
\noindent In this section, we wish to prove our main results. We begin by an observation an eigenvalues of $\mathcal{F}(n,0)$.

\begin{lemma}\label{ziro}
The eigenvalues of $\mathcal{F}(n,0)$ are given by
$$\eta_\lambda(0)=\sum_{k=0}^n(-1)^{n-k}\sum_{D\in \varepsilon(\lambda/(k))}\prod_{u\in D}h_{\lambda}(u)$$
where $\lambda$ runs over all partitions of $n$.
\end{lemma}

\begin{proof}
We have,
\begin{align}
\eta_\lambda(0)&=\sum_{k=0}^n(-1)^{n-k}\frac{n!}{(n-k)!}\frac{f^{\lambda/(k)}}{f^\lambda}\tag{Lemma \ref{rentel}}\\
&=\sum_{k=0}^n(-1)^{n-k}\frac{n!}{(n-k)!}\frac{|\lambda/(k)|!\sum_{D\in \varepsilon(\lambda/(k))}\prod_{u\in [\lambda]\backslash D}\frac{1}{h_\lambda(u)}}{\frac{n!}{\prod_{u\in [\lambda]} h_\lambda(u)}} \tag{Lemmas \ref{hlf}, \ref{nhlf}}\\
&=\sum_{k=0}^n(-1)^{n-k}\frac{n!}{(n-k)!}\frac{\frac{(n-k)!}{{\prod_{u\in [\lambda]} h_\lambda(u)}}\sum_{D\in \varepsilon(\lambda/(k))}\prod_{u\in  D}h_\lambda(u)}{\frac{n!}{\prod_{u\in [\lambda]} h_\lambda(u)}}\nonumber\\
&=\sum_{k=0}^n(-1)^{n-k}\sum_{D\in \varepsilon(\lambda/(k))}\prod_{u\in D}h_{\lambda}(u).\nonumber
\end{align}
\end{proof}

\noindent\textit{\bf Proof of Theorem \ref{main}}.
Let $\lambda \vdash n$. Then
\begin{align}
\eta_\lambda(k)&=\frac{1}{f^\lambda}\sum_{\beta \in S(n,k)}[\lambda]\beta\tag{Lemma \ref{eigen}}\\
&=\frac{\binom{n}{k}}{f^\lambda}\sum_{\beta \in S(n-k,0)}[\lambda]\beta\nonumber\\
&=\frac{\binom{n}{k}}{f^\lambda}\sum_{\beta \in S(n-k,0)}\sum_{\substack{\mu \vdash n-k\\ \mu\subseteq \lambda}} f^{\lambda/\mu}[\mu]\beta\tag{Lemma \ref{product}}\\
&=\frac{\binom{n}{k}}{f^\lambda}\sum_{\substack{\mu \vdash n-k\\ \mu\subseteq \lambda}}f^\mu f^{\lambda/\mu}(\frac{1}{f^\mu}\sum_{\beta \in S(n-k,0)} [\mu]\beta)\nonumber\\
&=\frac{\binom{n}{k}}{f^\lambda} \sum_{\substack{\mu \vdash n-k\\ \mu\subseteq \lambda}}f^\mu f^{\lambda/\mu} \eta_\mu(0)\nonumber
\end{align}
\begin{align}
&=\frac{\binom{n}{k}\prod_{u\in [\lambda]} h_\lambda(u)}{n!} \sum_{\substack{\mu \vdash n-k\\ \mu\subseteq \lambda}}(\frac{(n-k)!}{\prod_{u\in [\mu]} h_\mu(u)}) (|\lambda/\mu|!\sum_{E\in \varepsilon(\lambda/\mu)}\prod_{u\in [\lambda]\backslash E}\frac{1}{h_\lambda(u)})\eta_\mu(0)\tag{Lemmas \ref{hlf}, \ref{nhlf}}\\
&=\frac{\binom{n}{k}\prod_{u\in [\lambda]} h_\lambda(u)}{n!} \sum_{\substack{\mu \vdash n-k\\ \mu\subseteq \lambda}}[(\frac{(n-k)!}{\prod_{u\in [\mu]} h_\mu(u)}) (\frac{k!}{\prod_{u\in [\lambda]} h_\lambda(u)}\sum_{E\in \varepsilon(\lambda/\mu)}\prod_{u\in E}h_\lambda(u))\eta_\mu(0)]\nonumber\\
&= \sum_{\substack{\mu \vdash n-k\\ \mu\subseteq \lambda}}[(\frac{1}{\prod_{u\in [\mu]} h_\mu(u)}) (\sum_{E\in \varepsilon(\lambda/\mu)}\prod_{u\in E}h_\lambda(u))\eta_\mu(0)] \nonumber\\
&= \sum_{\substack{\mu \vdash n-k\\ \mu\subseteq \lambda}}[\frac{\sum_{t=0}^{n-k}(-1)^{n-k-t}\sum_{D\in \varepsilon(\mu/(t))}\prod_{u\in D}h_{\mu}(u)}{\prod_{u\in [\mu]} h_\mu(u)} \sum_{E\in \varepsilon(\lambda/\mu)}\prod_{u\in E}h_\lambda(u)]. \tag{Lemma \ref{ziro}}
\end{align}
 Moreover, using Theorem \ref{eigen} and Lemma \ref{hlf}, the multiplicity of $\eta_{\lambda} (k)$ is
    $(\frac{n!}{{\prod}_{u\in [\lambda]} h_{\lambda}(u)})^2$.
This completes the proof.\qed\\

In the sequel of this section, we require an interesting observation as follows:
\begin{lemma}\label{good}
Let $n$ and $k$ be non-negative integers with $k \leq n$. Then for every $\lambda\vdash n$,
$$\binom{n}{k}=\sum_{\substack{\mu \vdash n-k\\ \mu \subseteq \lambda}}\frac{\sum_{E\in \varepsilon(\lambda/\mu)}\prod_{u\in E}h_{\lambda}(u)}{\prod_{u\in [\mu]}h_{\mu}(u)}.$$
\end{lemma}
\begin{proof}
suppose $A:=\sum_{\substack{\mu \vdash n-k\\ \mu \subseteq \lambda}}\frac{\sum_{E\in \varepsilon(\lambda/\mu)}\prod_{u\in E}h_{\lambda}(u)}{\prod_{u\in [\mu]}h_{\mu}(u)}$. Then
\begin{align}
A&=\sum_{\substack{\mu \vdash n-k\\ \mu \subseteq \lambda}}[\binom{n}{k}\frac{\prod_{u\in [\lambda]}h_{\lambda}(u)}{n!}\frac{(n-k)!}{\prod_{u\in [\mu]}h_{\mu}(u)}\frac{|\lambda/\mu|!}{\prod_{u\in [\lambda]}h_{\lambda}(u)}\sum_{E\in \varepsilon(\lambda/\mu)}\prod_{u\in E}h_{\lambda}(u)]\nonumber\\
&=\frac{\binom{n}{k}}{f^\lambda}\sum_{\substack{\mu \vdash n-k\\ \mu \subseteq \lambda}}f^{\lambda/\mu}f^\mu\tag{Lemmas \ref{hlf}, \ref{nhlf}}\\
&=\frac{\binom{n}{k}f^\lambda}{f^\lambda}=\binom{n}{k}.\tag{Lemma \ref{product}}
\end{align}
\end{proof}

\noindent\textit{\bf Proof of Corollary \ref{transposition}}
Let $\lambda \vdash n$. Then by Theorem \ref{main}
\begin{align}
\eta_\lambda(n-2)&= \sum_{\substack{\mu \vdash 2\\ \mu \subseteq \lambda}}[\frac{\sum_{t=0}^{2}(-1)^{2-t}\sum_{D\in \varepsilon(\mu/(t))}\prod_{u\in D}h_{\mu}(u)}{\prod_{u\in [\mu]}h_{\mu}(u)}\sum_{E\in \varepsilon(\lambda/\mu)}\prod_{u\in E}h_{\lambda}(u)]\nonumber\\
&= \frac{1}{2}(\sum_{E\in \varepsilon(\lambda/(2))}\prod_{u\in E}h_{\lambda}(u)-\sum_{E\in \varepsilon(\lambda/(1^2))}\prod_{u\in E}h_{\lambda}(u))
\tag{1}
\end{align}
Note that if $\varepsilon(\lambda/ \mu)=\emptyset$, then $\sum_{E\in \varepsilon(\lambda/\mu)}\prod_{u\in E}h_{\lambda}(u)=0$. Now using Lemma \ref{good},
\begin{align}
\binom{n}{2}&=\frac{1}{2}(\sum_{E\in \varepsilon(\lambda/(2))}\prod_{u\in E}h_{\lambda}(u)+\sum_{E\in \varepsilon(\lambda/(1^2))}\prod_{u\in E}h_{\lambda}(u)).\tag{2}
\end{align}
Hence equations (1) and (2) imply that
\begin{align}
\eta_\lambda(n-2)&=\sum_{E\in \varepsilon(\lambda/(2))}\prod_{u\in E}h_{\lambda}(u)-\binom{n}{2}\nonumber\\
&=\sum_{\substack{(i,i),(j,j+1)\in [\lambda]\\ i\leq j}}h_\lambda((i,i))h_\lambda((j,j+1))-\binom{n}{2}\nonumber
\end{align}
Therefore using Theorem \ref{main}, we are done.\qed\\

\noindent\textit{\bf Proof of Corollary \ref{baund}}
\begin{align}
a)\; |\eta_\lambda(k)|&= |\sum_{\substack{\mu \vdash n-k\\ \mu \subseteq \lambda}}[\frac{\sum_{t=0}^{n-k}(-1)^{n-k-t}\sum_{D\in \varepsilon(\mu/(t))}\prod_{u\in D}h_{\mu}(u)}{\prod_{u\in [\mu]}h_{\mu}(u)}\sum_{E\in \varepsilon(\lambda/\mu)}\prod_{u\in E}h_{\lambda}(u)]|\tag{Theorem \ref{main}}\\
&\leq \sum_{\substack{\mu \vdash n-k\\ \mu \subseteq \lambda}}[\frac{|\sum_{t=0}^{n-k}(-1)^{n-k-t}\sum_{D\in \varepsilon(\mu/(t))}\prod_{u\in D}h_{\mu}(u)|}{\prod_{u\in [\mu]}h_{\mu}(u)}\sum_{E\in \varepsilon(\lambda/\mu)}\prod_{u\in E}h_{\lambda}(u)]\nonumber\\
&=\sum_{\substack{\mu \vdash n-k\\ \mu \subseteq \lambda}}[\frac{|\eta_{\mu}(0)|}{\prod_{u\in [\mu]}h_{\mu}(u)}\sum_{E\in \varepsilon(\lambda/\mu)}\prod_{u\in E}h_{\lambda}(u)]\tag{Lemma \ref{ziro}}\\
&\leq \sum_{\substack{\mu \vdash n-k\\ \mu \subseteq \lambda}}[\frac{M_{k}(\lambda)}{\prod_{u\in [\mu]}h_{\mu}(u)}\sum_{E\in \varepsilon(\lambda/\mu)}\prod_{u\in E}h_{\lambda}(u)]\nonumber \\
&=\binom{n}{k}M_{k}(\lambda).\tag{Lemma \ref{good}}
\end{align}
b) Clearly, $\mathcal{F}(n,k)$ is vertex-transitive, so it is $|S(n,k)|$-regular and the largest eigenvalue of $\mathcal{F}(n,k)$ is $|S(n,k)|$. Let $\lambda \vdash n$. Then
\begin{align}
\eta_{\lambda}(k)&= \sum_{\substack{\mu \vdash n-k\\ \mu \subseteq \lambda}}[\frac{\sum_{t=0}^{n-k}(-1)^{n-k-t}\sum_{D\in \varepsilon(\mu/(t))}\prod_{u\in D}h_{\mu}(u)}{\prod_{u\in [\mu]}h_{\mu}(u)}\sum_{E\in \varepsilon(\lambda/\mu)}\prod_{u\in E}h_{\lambda}(u)]\tag{Theorem \ref{main}}\\
&=\sum_{\substack{\mu \vdash n-k\\ \mu \subseteq \lambda}}[\frac{\eta_{\mu}(0)}{\prod_{u\in [\mu]}h_{\mu}(u)}\sum_{E\in \varepsilon(\lambda/\mu)}\prod_{u\in E}h_{\lambda}(u)]\tag{Lemma \ref{ziro}}\\
&\geq \sum_{\substack{\mu \vdash n-k\\ \mu \subseteq \lambda}}[\frac{(\frac{-|S(n-k,0)|}{n-k-1})}{\prod_{u\in [\mu]}h_{\mu}(u)}\sum_{E\in \varepsilon(\lambda/\mu)}\prod_{u\in E}h_{\lambda}(u)]\tag{Lemma \ref{ku}}\\
&=\frac{-|S(n-k,0)|}{n-k-1}\sum_{\substack{\mu \vdash n-k\\ \mu \subseteq \lambda}}\frac{\sum_{E\in \varepsilon(\lambda/\mu)}\prod_{u\in E}h_{\lambda}(u)}{\prod_{u\in [\mu]}h_{\mu}(u)}\nonumber\\
&=\frac{-|S(n-k,0)|\binom{n}{k}}{n-k-1}\tag{Lemma \ref{good}}\\
&=\frac{-|S(n,k)|}{n-k-1}.\nonumber
\end{align}
c) Applying Theorems \ref{eigen} and \ref{main}, we are done.\\

\noindent d) Using part (c), it is easy to see that $\eta_{(1^n)}(k)=\frac{-|S(n,k)|}{n-k-1}$. Thus by part (b), $\eta:=\frac{-|S(n,k)|}{n-k-1}$ is the least eigenvalue of $\mathcal{F}(n,k)$.\\

\noindent e) Let $\mu \vdash n-k$ such that $\mu \subseteq \lambda$. We can see that $\varepsilon(\lambda/ \mu)=\{\mu\}$. Also it is clear that
$$\{\mu \vdash n-k|\;\mu \subseteq \lambda\}\subseteq\{\mu_s:=(s, 1^{n-k-s})|\;s\in \mathbb{N}_0,\;1\leq s \leq n-k\}, $$
where $(n-k)=(n-k,1^0)$ and $(1^{n-k})=(1,1^{n-k-1})$. Since $n-k-s \leq n-m$ and $s\leq min\{m,n-k\}$, we deduce that $m-k\leq s\leq min\{m,n-k\}$. Hence as $s$ is a positive integer, $max\{1,m-k\}\leq s\leq min\{m,n-k\}$.  Now using Theorem \ref{main}, we have:
\begin{align}
\eta_{\lambda}(k)&= \sum_{\mu \vdash n-k,\; \mu \subseteq \lambda}[\frac{\sum_{t=0}^{n-k}(-1)^{n-k-t}\sum_{D\in \varepsilon(\mu/(t))}\prod_{u\in D}h_{\mu}(u)}{\prod_{u\in [\mu]}h_{\mu}(u)}\sum_{E\in \varepsilon(\lambda/\mu)}\prod_{u\in E}h_{\lambda}(u)]\nonumber\\
&= \sum_{s=max\{1,m-k\}}^{min\{m,n-k\}}[\frac{\sum_{t=0}^{n-k}(-1)^{n-k-t}\sum_{D\in \varepsilon(\mu_s/(t))}\prod_{u\in D}h_{\mu_s}(u)}{\prod_{u\in [\mu_s]}h_{\mu_s}(u)}\prod_{u\in [\mu_s]}h_{\lambda}(u)]\nonumber
\end{align}
\begin{align}
&= \frac{(-1)^{n-k}}{(n-k)}\sum_{s=max\{1,m-k\}}^{min\{m,n-k\}}[\frac{1+\sum_{t=1}^{s}(-1)^t\prod_{u\in [(t)]}h_{\mu_s}(u)}{(s-1)!(n-k-s)!}\frac{n(m-1)!(n-m)!}{(m-s)!(k+s-m)!}]\nonumber\\
&= \frac{(-1)^{n-k}n(m-1)!(n-m)!}{n-k}\sum_{s=max\{1,m-k\}}^{min\{m,n-k\}}\nonumber\\
&\frac{1+(n-k)\sum_{t=1}^{s}\frac{(-1)^t(s-1)!}{(s-t)!}}{(s-1)!(n-k-s)!(m-s)!(s+k-m)!}\nonumber
\end{align}
Since
$$\frac{(-1)^{n-k}n(m-1)!(n-m)!}{n-k}=\frac{(-1)^{n-k}\binom{n}{k}(n-k-1)!k!}{\binom{n-1}{m-1}},$$
we have
\begin{align}
\eta_\lambda(k)&=\frac{(-1)^{n-k}\binom{n}{k}}{\binom{n-1}{m-1}}\sum_{s=max\{1,m-k\}}^{min\{m,n-k\}}[(1\nonumber\\
&+(n-k)\sum_{t=1}^{s}\frac{(-1)^t(s-1)!}{(s-t)!}){\binom{n-k-1}{s-1}}{\binom{k}{m-s}}]\nonumber\\
&=\frac{(-1)^{n-k}\binom{n}{k}}{\binom{n-1}{m-1}}\sum_{s=max\{1,m-k\}}^{min\{m,n-k\}}[(\frac{s}{n-k}+\sum_{t=1}^{s}(-1)^tt!\binom{s}{t}){\binom{n-k}{s}}{\binom{k}{m-s}}]\nonumber\\
&=\frac{(-1)^{n-k}\binom{n}{k}}{\binom{n-1}{m-1}}\sum_{s=max\{1,m-k\}}^{min\{m,n-k\}}[(\frac{s}{n-k}+(-1)^s|S(s,0)|-1){\binom{n-k}{s}}{\binom{k}{m-s}}]\nonumber\\
&=\frac{(-1)^{n-k}\binom{n}{k}}{\binom{n-1}{m-1}}\sum_{s=max\{1,m-k\}}^{min\{m,n-k\}}[((-1)^s|S(s,0)|-\frac{n-k-s}{n-k}){\binom{n-k}{s}}{\binom{k}{m-s}}]\nonumber
\end{align}
This completes the proof. \qed


\section*{Acknowledgements}
This research was supported in part
by a grant  from School of Mathematics, Institute for Research in Fundamental Sciences (IPM).


\end{document}